\documentclass[draft,onecolumn]{IEEEtran}
\usepackage{amssymb}
\usepackage{amsfonts}
\usepackage{setspace}
\usepackage[colorlinks,linkcolor=black,citecolor=black,urlcolor=black,%
            bookmarks=false]{hyperref}
\usepackage[ansinew]{inputenc}
\usepackage[british]{babel}
\usepackage[reqno]{amsmath}
\usepackage{amsthm,amsfonts,amssymb}
\usepackage[mathscr]{euscript}
\usepackage{graphicx,color}
\usepackage{multicol,enumerate,fancyhdr,emptypage,xspace,subfig}

\newtheorem{Thm}{Theorem}
\newtheorem{Def}[Thm]{Definition}

\newtheorem{Lem}[Thm]{Lemma}
\newtheorem{Cor}[Thm]{Corollary}

\newtheorem{Alg}[Thm]{Algorithm}

\newtheorem{Ex}[Thm]{Example}
\newtheorem{Prop}[Thm]{Proposition}
\newtheorem{Rem}[Thm]{Remark}

\newcommand{\dubbel}[1]{{\mathbb #1}}
\newcommand{\Z}{{\mathbb Z}}
\newcommand{\ZZ}{\Z_{p^r}}
\newcommand{\ZZN}{\Z_{p^r}^n}
\newcommand{\ZZD}{\Z_{p^r}[D]}
\newcommand{\ZZND}{\Z_{p^r}^n[D]}
\newcommand{\A}{\mathcal{A}_{p}}

\newcommand{\pf}{\textbf{Proof }}
\newcommand{\pfend}{\par\vspace{2ex}\noindent}
\newcommand{\eind}{\hspace*{\fill}$\Box$\par\vspace{2ex}\noindent}
\newcommand{\C}{{\mathcal C}}

\newcommand{\wt}{\mathrm{wt}}

\newcommand{\ie}{\emph{i.e.}}
\newcommand{\pb}{\boldmath $p$}

\newcommand{\im}{\textnormal{Im}}

\begin{document}

\title{Column distance of convolutional codes over $\Z_{p^r}$ %
\thanks{This work was supported by Portuguese funds through the CIDMA -- Center for Research and Development in Mathematics and Applications, and the Portuguese Foundation for Science and Technology (FCT-Funda\c{c}\~{a}o para a Ci\^{e}ncia e a Tecnologia), within project PEst-UID/MAT/04106/2013.}}

\author{
Diego Napp
\quad
Raquel Pinto
\quad
Marisa Toste
\thanks{Diego Napp and Raquel Pinto and Marisa Toste are with CIDMA -- Center for Research and Development in Mathematics and Applications.
  Department of Mathematics.
  University of Aveiro.
  Campus Universitario de Santiago,
  3810-193 Aveiro,
  Portugal}
}

\maketitle

\begin{abstract}
Rosenthal \emph{et al.} introduced and thoroughly studied the notion of Maximum Distance Profile (MDP) convolutional codes over (non-binary) finite fields refining the classical notion of optimum distance profile, see for instance \cite[p.164]{jo99}. These codes have the property that their column distances are maximal among all codes of the same rate and the same degree. In this paper we aim at studying this fundamental notion in the context of convolutional codes over a finite ring. We extensively use the notion of $p$-encoder to present upper-bounds on the column distances which allow to introduce the notion of MDP in the context of finite rings. A constructive method for (non necessarily free) MDP convolutional codes over $\ZZ$ is presented.

\end{abstract}


\section{Introduction}

The development of the theory of convolutional codes over rings stems from the fact that these codes are the most appropriate class of codes for phase modulation. Massey and Mittelholzer \cite{massey89} were the first to introduce linear convolutional codes over the residue class ring $\Z_M$, $M$ a positive integer and showed how these codes behave very different from convolutional codes over finite fields. Fundamental results of the structural properties of convolutional codes over finite rings can be found in, for instance, \cite{fa01,jo98}. Fagnani and Zampieri \cite{fa01} studied the theory of convolutional codes over the ring $\ZZ$ in the case the input sequence space is a free module. The problem of deriving minimal encoders (left prime and row-reduced) was posed by Solé \emph{et al.} in \cite{so07} and solved by Kuijper \emph{et al.} in \cite{ku09,ku07} using the concept of minimal $p$-encoder, which is an extension of the concept of $p$-basis introduced in \cite{va96} to the polynomial context.

The search for and design of good convolutional codes over $\ZZ$ has been investigated in several works in the literature. Unit-memory convolutional codes over $\Z_4$ that gives rise to binary trellis codes with high free distances together with several concrete constructions of these codes were reported in \cite{as94,ko95}. In \cite{jo98a} two 16-state trellis codes of rate $\frac{2}{4}$, again over $\Z_4$, were found by computer search. Also worth mentioning are the papers of \cite{so07,sovi_ISIT07} were convolutional codes achieving the Gilbert-Varshamov bound were presented. However, in contrast to the block code case \cite{bebusi16,gu12,no01}, little is known about distance properties and constructions of convolutional codes over large rings.

Recently, in \cite{el13}, a bound on the free distance of convolutional codes over $\Z_{p^r}$ was derived, generalizing the bound given in \cite{ro99a1} for convolutional codes over finite fields. Codes achieving such a bound were called Maximal Distance Separable (MDS). The concrete constructions of MDS convolutional codes over $\Z_{p^r}$ presented in \cite{el13} were restricted to \emph{free} codes and general constructions were built in \cite{NaPiTo16}.

Column distances of convolutional codes over finite fields have been already studied for decades \cite{co69,jo99}. However, the notion of Maximum Distance Profile (MDP) convolutional codes over (non-binary) finite fields have been defined and fully studied by Rosenthal \emph{et al.} in \cite{gl03,Hutchinson2005,Hutchinson2008b,to12}. These codes are characterized by the property that their column distances increase as rapidly as possible for as long as possible. Obviously, fast growth of the column distances is an important property for codes to be used with sequential decoding. For this reason these codes are very appealing for applications (see \cite{to12}) as the maximal possible growth
in the column distances means that these codes have the potential to have a maximal number of errors corrected per time interval. Despite the importance of the notion, column distances of convolutional code over a finite ring are yet unexplored.

In this paper we aim at investigating this notion. In particular, we derive upper-bounds on the column distances and provide explicit novel constructions of (non necessarily free) MDP convolutional codes over $\ZZ$. In the proof of these results, an essential role is played by the theory of $p$-basis and in particular of a canonical form of the $p$-encoders. As for the construction of MDP, in contrast with the papers \cite{no01,el13} where the Hensel lift of a cyclic code was used, in this paper a direct lifting is employed to build MDP convolutional codes over $\ZZ$ from known constructions of MDP convolutional codes over $\Z_p$. Note that by the Chinese Remainder Theorem, results on codes over $\Z_{p^r}$ can be extended to codes over $\Z_M$, see also \cite{carriegos2017,ch94,jo98,McD74}.

The paper is organized as follows: In the next section we introduce some preliminaries on  $p$-basis of $\Z_{p^r}[D]$-submodules of $\Z_{p^r}^n[D]$. After presenting block codes over $\Z_{p^r}$ we introduce the new notions of $p$-standard form and r-optimal parameters. We conclude the preliminaries by defining convolutional codes over $\ZZ$. In section \ref{sec:coldis} we define and study column distances of convolutional codes over $\ZZ$.  Finally, in Section \ref{sec:constructions} we propose a method to build MDP convolutional codes over $Z_{p^r}$. The most technical proofs of our results are in Section \ref{appendix}.

\vspace{3cm}
\section{Preliminaries}
In this section we establish the necessary notions and results needed in order to derive the main results of the paper. Some of these are known in the literature and others are new.

\subsection{$P$-basis and $p$-dimension}

Let $p$ be a prime integer. Any element in $\dubbel{Z}_{p^r}$ can be written uniquely as a linear combination of $1,p,p^2,\dots$ $ \dots, p^{r-1}$, with coefficients in $\mathcal{A}_p=\{0,1, \dots,p-1\} $ (called the $p$-adic expansion of the element) \cite{ca00a}. Note that all elements of $\mathcal{A}_p \backslash\{0\}$ are units. The notion of $p$-basis for $\ZZ$-submodules over $\ZZN$ were first presented in \cite{va96} and later extended for the module $\ZZND$ in \cite{ku07}. These notions will play an important role throughout the paper since they will allow us to analyse the distance properties of convolutional codes over $\dubbel{Z}_{p^r}[D]$.

Let $v_1(D), \dots, v_k(D) $ be in $\dubbel{Z}^n_{p^r}[D]$. The vector $\displaystyle \sum_{j=1}^k a_j(D) v_j(D)$,
with $a_j(D) \in \mathcal{A}_p [D]$, is said to be a {\bf $\boldsymbol{p}$-linear combination} of ${v_1(D), \dots, v_k(D)}$ and the set of all $p$-linear combinations of ${v_1(D), \dots, v_k(D)}$ is called the {\bf $\boldsymbol{p}$-span} of $\{v_1(D), \dots, v_k(D) \}$, denoted by
$p$-span $(v_1(D), \dots, v_k(D))$. An ordered set of vectors $(v_1(D), \dots, v_k(D))$ in $\dubbel{Z}^n_{p^r}[D]$ is said to be a {\bf $\boldsymbol{p}$-generator sequence} if
$p \, v_i(D)$ is a $p$-linear combination of $v_{i+1}(D), \dots, v_k(D)$, $\; \; i=1, \dots, k-1$,
and $p \, v_k(D)=0$.

If $(v_1(D), \dots, v_k(D))$ is a $p$-generator sequence,
$
p\mbox{-span}(v_1(D), \dots, v_k(D))=\mbox{span}(v_1(D), \dots, v_k(D))
$
\cite{ku07}  and consequently the  $p$-span$(v_1(D), \dots, v_k(D))$ is a $\ZZ$-submodule of $\dubbel{Z}^n_{p^r}[D]$. Moreover, note that if $M=\mbox{span}(v_1(D), \dots, v_k(D))$,
\begin{equation}\label{eq:07}
(v_1(D), pv_1(D) \dots,  p^{r-1}v_1(D),v_2(D), pv_2(D), \dots, p^{r-1}v_2(D), \dots, v_l(D), pv_k(D) \dots, p^{r-1}v_k(D))
\end{equation}
is a $p$-generator sequence of $M$.


The vectors $v_1(D), \dots, v_k(D)$  in $\dubbel{Z}_{p^r}^n[D]$ are said to be {\bf $\boldsymbol{p}$-linearly independent} if the only $p$-linear combination of
$v_1(D), \dots, v_k(D)$ that is equal to $0$ is the trivial one.

An ordered set of vectors $(v_1(D), \dots, v_k(D))$ which is a $p$-generator sequence of $M$ and $p$-linearly independent is said to be a {\bf $\boldsymbol{p}$-basis} of $M$. It is proved in \cite{ku09} that two $p$-bases of a $\dubbel{Z}_{p^r}$-submodule $M$ of $\dubbel{Z}_{p^r}^n[D]$ have the same number of elements. This number of elements is called {\bf $\boldsymbol{p}$-dimension} of $M$.

A nonzero polynomial vector $v(D)$ in $\dubbel Z_{p^r}^n[D]$, written as $v(D)=\sum\limits_{t=0}^{\nu}{v_tD^t}$, with $v_t \in \mathbb Z_{p^r}^n$, and $v_{\nu} \neq 0$, is said to have degree $\nu$, denoted by $deg \, v(D)=\nu$, and $v_{\nu}$ is called the {\bf leading coefficient vector} of $v(D)$, denoted by $v^{lc}$. For a given matrix $G(D)\in \dubbel{Z}_{p^r}^{k \times n}[D]$ we denote by $G^{lc} \in \dubbel{Z}_{p^r}^{k \times n}$ the matrix whose rows are constituted by the leading coefficient of the rows of $G(D)$. A $p$-basis $(v_1(D), \dots, v_k(D))$ is called a {\bf reduced $\boldsymbol{p}$-basis} if the vectors $v_1^{lc}, \dots, v_k^{lc}$ are $p$-linearly independent in $\dubbel Z_{p^r}$, see also \cite{ku11}. 

Every submodule $M$ of $\ZZND$ has a reduced $p$-basis. Algorithm 3.11 in~\cite{ku07} constructs a reduced $p$-basis for a submodule $M$ from a generator sequence of $M$.
 The degrees of the vectors of  two reduced $p$-bases of $M$ are the same (up to permutation) and their sum is called the {\bf $\boldsymbol{p}$-degree} of $M$. 

\subsection{Block codes over a finite ring}


A {\bf (linear) block code $\C$} of length $n$ over $\ZZ$ is a $\ZZ$-submodule of $\ZZN$ and the elements of $\C$ are called codewords.
A {\bf generator matrix} $\widetilde G\in \ZZ^{\widetilde k \times n}$ of $\C$ is a matrix whose rows form a minimal set of generators of $\C$ over $\ZZ$. If $\widetilde G$ has full row rank, then it is called an {\bf encoder} of $\C$ and $\C$ is a free module.
If $\C$ has $p$-dimension $k$, a {\bf \pb-encoder} $G \in \ZZ^{k \times n}$ of $\C$ is a matrix whose rows form a $p$-basis of $\C$ and therefore
\begin{equation*}
  \begin{split}
     \C & = \im_{\A}G  = \{v=uG \in \ZZN \ : u \in \mathcal{A}^k_p\}.
  \end{split}
\end{equation*}


Let $\C$ be a block code over $\ZZ$. A generator matrix $\widetilde{G}$ for $\C$ is said to be in {\bf standard form} if
\begin{equation} \label{stand}
  \widetilde{G} = \left[\!\!\!\!
                         \begin{array}{ccccccc}
                            I_{k_0} & A^0_{1,0} & A^0_{2,0} & A^0_{3,0} & \cdots & A^0_{r-1,0} & A^0_{r,0} \\
                            0 & pI_{k_1} & pA^1_{2,1} & pA^1_{3,1} & \cdots & pA^1_{r-1,1} & pA^1_{r,1} \\
                            0 & 0 & p^2I_{k_2} & p^2A^2_{3,2} & \cdots & p^2A^2_{r-1,2} & p^2A^2_{r,2} \\
                            \vdots&\vdots&\vdots&\vdots&\ddots&\vdots&\vdots\\
                            0 & 0 & 0 & 0 & \cdots & p^{r-1}I_{k_{r-1}} & p^{r-1}A^{r-1}_{r,r-1} \\
                         \end{array}
                   \!\!\!\! \right],
\end{equation}

\flushleft
where $I_{k_i}$ denotes the identity matrix of size $k_i$ and the columns are grouped into blocks with $k_0, \dots, k_{r-1}$ and $n-\sum_{i=0}^{r-1}{k_i}$ columns.

\begin{Lem} \label{norton} \cite[Theorem 3.3.]{no01}
  Any nonzero block code $\C$ over $\ZZ$ has, after a suitable permutation of the coordinates, a generator matrix in standard form.
  Moreover, all generator matrices of $\C$ in standard form have the same parameters $k_0, k_1, \dots, k_{r-1}$.
\end{Lem}

Note that a block code over $\ZZ$ is free if and only if the parameters of any generator matrix in standard form are $k_0=\tilde{k}$, $k_i=0$, $i=1, \dots, r-1$. For the purposes of this work, we introduce a novel canonical form that can be considered as the $p$-analog of the standard form for $p$-encoders. 

Given a $p$-basis $(v_1, \dots, v_k)$ of $\mathcal{C}$ there are certain operations that can be applied to $(v_1, \dots, v_k)$ so that we obtain another $p$-basis of $\mathcal{C}$. Some of these elementary operations are described in the following lemma which is not difficult to prove, see more details in \cite{TosteTese}.

\begin{Lem} \label{oper}
Let $(v_1, \dots, v_k)$ be a $p$-basis of a submodule $M$ of $\dubbel{Z}^n_{p^r}$. Then,
\begin{enumerate}
  \item If $v'_i=v_i+\sum_{j=i+1}^{k}a_jv_j$, with $a_j \in \mathcal{A}_{p^r}$, then $(v_1, \dots, v_{i-1}, v'_i, v_{i+1}, \dots, v_k)$ is a $p$-basis of $M$.
  \item If $pv_i$ is a $p$-linear combination of $v_j,v_{j+1}, \dots, v_k$, for some $j>i$, then $(v_1, \dots, v_{i-1}, v_{i+1}, \dots, v_{j-1}, v_i, v_j, \dots, v_k)$ is a $p$-basis of $M$.
\end{enumerate}
\end{Lem}

A generator matrix $\widetilde G$ of $\C$ in standard form as in (\ref{stand}) can be extended (see algorithm below) to obtain \emph{a} $p$-encoder $G$ in the following form:
{
\fontsize{9}{9}\selectfont
\arraycolsep=3pt 
\begin{equation}\label{p-standG0}
 \left[\!\!\!\!
  \begin{array}{ccccccc}
    I_{k_0} & A^0_{1,0} & A^0_{2,0} & A^0_{3,0} & \cdots & A^0_{r-1,0} & A^0_{r,0} \\
    ------&------&------&------&------&------&------ \\
    pI_{k_0} & 0 & pA^0_{2,1} & pA^0_{3,1} & \cdots & pA^0_{r-1,1} & pA^0_{r,1} \\
    0 & pI_{k_1} & pA^1_{2,1} & pA^1_{3,1} & \cdots & pA^1_{r-1,1} & pA^1_{r,1} \\
    ------&------&------&------&------&------&------ \\
    p^2I_{k_0} & 0 & 0 & p^2A^0_{3,2} & \cdots & p^2A^0_{r-1,2} & p^2A^0_{r,2} \\
    0 & p^2I_{k_1} & 0 & p^2A^1_{3,2} & \cdots & p^2A^1_{r-1,2} & p^2A^1_{r,2} \\
    0 & 0 & p^2I_{k_2} & p^2A^2_{3,2} & \cdots & p^2A^2_{r-1,2} & p^2A^2_{r,2} \\
    ------&------&------&------&------&------&------ \\
    \vdots&\vdots&\vdots&\vdots&\cdots&\vdots&\vdots\\
    ------&------&------&------&------&------&------ \\
    p^{r-1}I_{k_0} & 0 & 0 & 0 & \cdots & 0 & p^{r-1}A^0_{r,r-1} \\
    0 & p^{r-1}I_{k_1} & 0 & 0 & \cdots & 0 & p^{r-1}A^1_{r,r-1} \\
    0 & 0 & p^{r-1}I_{k_2} & 0 & \cdots & 0 & p^{r-1}A^2_{r,r-1} \\
    0 & 0 & 0 & p^{r-1}I_{k_3} & \cdots & 0 & p^{r-1}A^3_{r,r-1} \\
    \vdots&\vdots&\vdots&\vdots&\ddots&\vdots&\vdots\\
    0 & 0 & 0 & 0 & \cdots & p^{r-1}I_{k_{r-1}} & p^{r-1}A^{r-1}_{r,r-1} \\
  \end{array}
\!\!\!\! \right].
\end{equation}
}One can verify that the scalars $k_i$, $i=0,1, \dots, r-1$, are equal for all $p$-encoders of $\C$ in this form, \emph{i.e.}, they are uniquely determined for a given code $\mathcal{C} \subset \dubbel{Z}_{p^r}^n$ and coincide with the parameters appearing in (\ref{stand}) for generator matrices in standard form. We call $k_0, k_1, \dots, k_{r-1}$ the {\bf parameters} of $\mathcal{C}$. If $G$ is in such a form we say that $G$ is in the {\bf $\boldsymbol{p}$-standard form}. The $p$-standard form will be a useful tool to prove our results in the same way the standard form was for previous results in the literature, see for instance \cite{ca00a,no01}.

%

For completeness we include the following straightforward algorithm that transforms a given generator matrix $\widetilde G$ of $\C$ in standard form  into \emph{a} $p$-encoder $G$ in $p$-standard form.

\begin{Alg}\label{alg_p-enc}
\hspace{1cm} \underline{{\em Input data:}} $ \widetilde{G} \leftarrow \left[\!\!\!\!
             \begin{array}{c}
                B_{1,k_0}\\
                pB_{1,k_1}\\
                \vdots\\
                p^{r-1}B_{1,k_{r-1}}\\
             \end{array}
       \!\!\!\! \right]$
generator matrix in standard form, \emph{i.e.}, as in (\ref{stand}), where $B_{1,k_i}\in \ZZ^{k_i \times n}$, for $i=0,1, \dots, r-1$.

\begin{description}
  \item[{\bf \underline{Step 1}}: ] \ \ Expand $\widetilde{G}$ multiplying $p^iB_{1,k_i}$ by $p, p^2, \dots, p^{r-(i+1)} $, with $i=0, \dots, r-2$, resulting in
   $$
  G \leftarrow \left[\!\!\!\!
             \begin{array}{c}
                B'_{1,k_0}\\
                B'_{2,k_0}\\
                \vdots\\
                B'_{r,k_0}\\
                ---\\
                B'_{1,k_1}\\
                \vdots\\
                B'_{r-1,k_1}\\
                ---\\
                \vdots\\
                ---\\
                B'_{1,k_{r-1}}\\
             \end{array}
       \!\!\!\! \right],
  $$
  where $B'_{j,k_i}=p^{i+j-1}B_{1,k_i}$, $j=1, \dots, r-i$, $i=0, \dots, r-1$.

  \item[{\bf \underline{Step 2}}:] \ \  For $j=2, \dots, r-i$ and $i=0, \dots, r-2$ replace
$$
B'_{j,k_i} \rightarrow B'_{j,k_i}-\sum_{t=1}^{j-1}{A^i_{i+t,i}B'_{j-t,k_{i+t}}}.
$$

   \item[{\bf \underline{Step 3}}:] \ \ Reorder the rows in order to have $G$ written in $p$-standard form.

\end{description}

\underline{{\em Output data}}: $G$.
\end{Alg}

\begin{Thm} \label{result_ALG}
Given a generator matrix  $\widetilde{G}$ in standard form as in (\ref{stand}) of a block code $\C$ over $\ZZ$, the Algorithm \ref{alg_p-enc} produces a $p$-encoder $G$ of $\C$ in $p$-standard form. Moreover, if $\C$ has $p$-dimension $k$ then $k= \sum_{i=0}^{r-1} k_i (r-i)$.
\end{Thm}

\begin{proof}
From (\ref{eq:07}) we guarantee that, in Step 1  we construct a $p$-generator sequence of $\C$.
The structure of $\widetilde{G}$ defined in (\ref{stand}) allows to state immediately that the rows of ${G}$ are $p$-linearly independent and, therefore ${G}$ is a $p$-encoder of $\C$.
By Lemma \ref{oper}, Step 2 and Step 3 produce a $p$-encoder in $p$-standard form. This, together with Lemma \ref{norton} immediately implies last statement.
\end{proof}

The {\bf free distance $d(\C)$} of a linear block code $\C$ over $\ZZ$ is given by
$$
d(\C)= \min \{\wt (v), v \in \C, v \neq 0 \}
$$
where $\wt (v)$ is the Hamming weight of $v$, \ie, the number of nonzero entries of $v$.

Since the last row of a $p$-encoder (or of a generator matrix in standard form) in $p$-standard form is obviously a codeword we can easily recover the Singleton-type upper bound on the free distance of a block code over $\ZZ$ derived in \cite{no01}.

\begin{Thm}
Given a linear block code $\C \subset \ZZN$ with parameters $k_0,\dots,k_{r-1}$, it must hold that
$$
d(\C)\leq n-(k_0 + \dots + k_{r-1})+1.
$$
\end{Thm}

Among block codes of length $n$ and $p$-dimension $k$, we are interested in the ones with largest possible distance.
For that we need to introduce the notion of an optimal set of parameters of $M$ \cite{TosteTese}.

\begin{Def}
Given an integer $r\geq 1$ and a non-negative integer $k$ we call an ordered set $(k_0, k_1, \cdots, k_{r-1})$, $k_i \in \mathbb{N}_0, \ i=0, \cdots, r-1$ an {\bf $\boldsymbol{r}$-optimal set of parameters} of $k$ if
$$
 k_0+k_1+ \cdots+k_{r-1} = \min_{k=rk'_0+(r-1)k'_1+ \cdots+k'_{r-1}} (k'_0+k'_1+ \cdots+k'_{r-1}).
$$
\end{Def}

Note that when $r$ divides $k$, $(k_0,0,\dots, ,0)$, with $k_0=\frac{k}{r}$, is the unique $r$-optimal set of parameters of $k$. However, in general, the $r$-optimal set of parameters of $k$ is not necessarily unique for a given $k$ and $r$. For instance if $k=25$ and $r=6$, $(4,0,0,0,0,1)$ and $(0,5,0,0,0,0)$ are two possible $6$-optimal set of parameters of $25$. Note that the computation of the $r$-optimal set of parameters is the well-known change making problem \cite{ch70}.

\begin{Lem}\label{opt_parameters} \cite{NaPiTo16}
Let $(k_0, k_1, \cdots, k_{r-1})$ be an $r$-optimal set of parameters of $k$. Then,
$$
k_0+k_1+ \cdots+k_{r-1}=\left\lceil\frac{k}{r}\right \rceil.
$$
\end{Lem}


Hence, for a given $\C\subset \ZZ^n$ with $p$-dimension $k$, a Singleton bound can be defined.

\begin{Cor}
Given a block code $\C\subset \ZZN$ and $p$-dimension $k$,
$$
d(\C)\leq n- \left \lceil\frac{k}{r}\right \rceil +1.
$$
\end{Cor}

This bound also follows from the fact that, for any block code (not necessarily linear) we have that $|\C |\leq (p^r)^{n- d(\C) +1} $, see \cite{no01}, and it can be also find in \cite{el13}.

\subsection{Convolutional codes over a finite ring}

Next we will consider convolutional codes constituted by left compact sequences in $\ZZ$, \emph{i.e.,}  in which the elements of the code will be of the form
$$
\begin{array}{cccc}
w: & \mathbb Z & \rightarrow &  \ZZN \\
& t & \mapsto & w_t
\end{array}
$$
where $w_t=0$ for $t < \ell$ for some $\ell \in \mathbb Z$. These sequences can be represented by Laurent series,
$$
w(D)= \displaystyle \sum_{t=\ell}^\infty w_t D^t\in \ZZ((D)).
$$
Let us denote by $\ZZ(D)$ the \index{Ring of rational matrices} ring of rational matrices defined in $\ZZ$. More precisely, $\ZZ(D)$ is the set
$$
\{\frac{p(D)}{q(D)}: p(D),q(D) \in \ZZD \mbox{ and the coefficient of the smallest power of $D$ in $q(D)$ is a unit}\}.
$$
This last condition allows us to treat a rational function as an equivalence class in the relation
$$
\frac{p(D)}{q(D)} \sim \frac{p_1(D)}{q_1(D)} \mbox{ if and only if } p(D)q_1(D) = p_1(D) q(D).
$$

Note that $\ZZ(D)$ is a subring of $\ZZ((D))$ and, obviously $\ZZ[D]$ is a subring of $\ZZ(D)$.

\vspace{.2cm}

A rational matrix $A(D) \in \ZZ^{\ell \times \ell}(D)$ is invertible if there exists a rational matrix $L(D) \in \ZZ^{\ell \times \ell}(D)$ such that $L(D)A(D)=I$. 
Moreover, $A(D)$ is invertible if and only if $\bar{A}(D)$ is invertible in $\mathbb Z_p^{\ell \times \ell}(D)$, where $\bar{A}(D)$ represents the projection of $A(D)$ into $\mathbb Z_p(D)$.

Most of the literature on convolutional codes over rings considers codewords as elements in the ring of Laurent series \cite{fa96,forneyT93,jo98,ku09,loeligerM96,el13}. We shall adopt this approach and define a {\bf convolutional code $\mathcal{C}$} of length $n$ as a $\ZZ((D))$-submodule of  $\ZZN((D))$ for which there exists a polynomial matrix $\widetilde G(D) \in \ZZ^{\tilde k \times n}[D]$ such that
\begin{eqnarray*}
{\mathcal C} & = & \im_{\ZZ((D))} \widetilde G(D) =  \left\{u(D)\widetilde G(D) \in \ZZN((D)) :\, u(D) \in \Z^{\widetilde k}_p((D))\right\}.
\end{eqnarray*}

The matrix $ \widetilde G(D)$ is called a {\bf generator matrix} of $\C$. If $\widetilde G(D)$ is full row rank then it is called an {\bf encoder} of $\C$. Moreover, if
\begin{eqnarray*}
{\mathcal C}  & = & \im_{{\cal A}_{p}((D))} G(D) =  \left\{u(D)G(D) \in \ZZN((D)) :\, u(D) \in \mathcal{A}^k_p((D))\right\},
\end{eqnarray*}
where $\mathcal{A}_p((D))=\{\sum_{i=s}^{+\infty}{a_iD^i}: a_i \in \mathcal{A}_p \mbox{ and } s \in \mathbb Z\}$, and $G(D) \in \ZZ^{k \times n}[D]$ is a polynomial matrix whose rows form a $p$-basis, then we say that $G(D)$ is a {\bf $\boldsymbol{p}$-encoder} of $\mathcal{C}$ and $\C$ has {\bf $\boldsymbol{p}$-dimension} $k$.

\begin{Rem}
We emphasize that in this paper we do not assume that $\C$ is free. Hence, it is important to underline that there exists convolutional codes that do not admit an encoder. However, they always admit a $p$-encoder. For this reason the notion of $p$-encoder is more interesting and natural than the standard notion of the encoder. The difference is that the input vector takes values in $\mathcal{A}^k_p((D))$ for $p$-encoders whereas for generator matrices it takes values in $\Z^{\widetilde k}_p((D))$. This idea of using a p-adic expansion for the information input vector is already present in, for instance, \cite{ca00a} and was further developed in \cite{va96} introducing the notion of $p$-generator sequence of vectors in $\ZZ$. In \cite{ku09,ku07} this notion was extended to polynomial vectors.
\end{Rem}

Next, we present two straightforward results as lemmas.

\begin{Lem}
If $\widetilde{G}(D) \in \ZZ^{\widetilde{k} \times n} [D]$ is a generator matrix of a convolutional code $\C$ and $X(D) \in \ZZ^{\widetilde{k} \times \widetilde{k}}(D)$ is an invertible rational matrix such that $X(D) \widetilde{G}(D)$ is polynomial, then
$$
\im_{\ZZ((D))} \widetilde{G}(D) = \im_{\ZZ((D))} X(D) \widetilde{G}(D),
$$
which means that $X(D)\widetilde{G}D)$ is also a generator matrix of $\C$.
\end{Lem}

\begin{Lem} \cite{QuNaPiTo17}\label{rational encoder}
Let $\C$ be a $\ZZ((D))$-submodule of $\ZZ^n((D))$ given by $\C=\im_{\ZZ((D))}N(D)$, where $N(D) \in \ZZ^{\widetilde k \times n}(D)$. Then $\C$ is a convolutional code, and if $N(D)$ is full row rank, $\C$ is a free code of rank $\widetilde{k}$.
\end{Lem}


A generator matrix $\widetilde G(D) \in \mathbb Z_{p^r}^{\widetilde k \times n}[D]$ is said to be \emph{noncatastrophic} (\cite{ku09}) if for any $u(D) \in \mathbb Z_{p^r}^{\widetilde k}((D))$,
$$
u(D) \widetilde G(D) \in \ZZN[D] \implies u(D) \in \mathbb Z_{p^r}^{\widetilde k}[D].
$$
Note that this property is a characteristic of a generator matrix and not a property of the code. For example in $\mathbb Z_4$, $G_1(D)=[1+D \; \; 1+D]$ and $G_2(D)=[1 \;\; 1]$ are two encoders of the same convolutional code, but $G_2(D)$ is noncatastrophic and $G_1(D)$ is catastrophic. However, there are convolutional codes that
do not admit noncatastrophic generator matrices like shown in the following example \cite{ku09}.

\begin{Ex} \label{noncat} The convolutional code over $\mathbb Z_4$ with encoder $\widetilde G(D)= [1+D \;\; 1+3D]$ does not admit a noncatastrophic encoder.
\end{Ex}

Obviously, a generator matrix that is not full row rank is catastrophic and therefore convolutional codes that are not free do not admit noncatastrophic encoders.


Analogously, we say that a $p$-encoder $G(D) \in \ZZ^{k \times n}[D]$ is said to be noncatastrophic \cite{ku09} if for any $u(D) \in {\cal A}_p^k((D))$,
$$
u(D)G(D) \in \ZZN[D] \implies u(D) \in {\cal A}_p^k[D].
$$

If a convolutional code $\cal C$ admits a noncatastrophic encoder $\widetilde G(D) \in \ZZ^{\widetilde k \times n}[D]$ then it also admits a noncatastrophic $p$-encoder, namely
$$
G(D)=\left[
\begin{array}{c}
\widetilde G(D) \\
p \widetilde G(D)\\
\vdots \\
p^{r-1} \widetilde G(D)
\end{array}
\right].
$$
However, there are convolutional codes that do not admit noncatastrophic encoders but admit noncatastrophic $p$-encoders like it is shown in the next example.

\begin{Ex} Let us consider again the convolutional code $\cal C$ over $\mathbb Z_4$ of Example \ref{noncat}. The $p$-encoder $G(D)=\left[
\begin{array}{cc}
1+D & 1+3D \\
2 & 2
\end{array}
\right]$ of $\cal C$ is noncatastrophic.
\end{Ex}

We call a convolutional code that admits a noncatastrophic $p$-encoder a \textbf{noncatastrophic convolutional code}. Thus, the class of noncatastrophic convolutional codes contain the class of convolutional codes that admit a noncatastrophic encoder. In \cite{ku09} it was conjectured that all the convolutional codes admit a noncatastrophic $p$-encoder and this is still an open problem.

\vspace{.2cm}

Another property of $p$-encoders that is relevant for this work is ``\emph{delay-freeness}". We say that a $p$-encoder $G(D)$ of a convolutional code $\C$ is delay-free if for any $u(D) \in {\cal A}_p^k((D))$ and any $N\in \Z$
$$
\mbox{supp } (u(D)G(D)) \subset [N, + \infty) \implies \mbox{supp } (u(D)) \subset [N, + \infty),
$$
where $\mbox{supp } (v(D))$ denotes the support of $v(D)=\sum v_i D^i$, \emph{i.e.,} $\mbox{supp } (v(D))=\{i: v_i \neq 0\}$.

\begin{Lem} \cite{ku09}
Let $G(D) \in \ZZ^{k \times n}[D]$ be a $p$-encoder. Then $G(D)$ is delay-free if and only if the rows of $G(0)$ are $p$-linearly
independent in $\ZZ^n$.
\end{Lem}

All convolutional codes admit a delay-free $p$-encoder. Moreover, if $\cal C$ is a noncatastrophic convolutional code, then it admits a delay-free and noncatastrophic $p$-encoder which rows form a reduced $p$-basis \cite{ku09}.

Let $\cal C$ be a noncatastrophic convolutional code of length $n$ over $\ZZ$ and let $G(D)$ be a delay-free noncatastrophic $p$-encoder of $\C$, such that its rows form a reduced $p$-basis. Then $G(D)$ is called a\textbf{ minimal $p$-encoder} of $\cal C$. The degrees of the rows of $G(D)$ are called the \textbf{$p$-indices} of $\cal C$ and the \textbf{$p$-degree} of $\cal C$ is defined as the sum of the $p$-indices of $\cal C$. Moreover, if $\cal C$ has $p$-dimension $k$ and $p$-degree $\delta$, $\cal C$ is called an $(n,k,\delta)$-convolutional code.

\section{Column distance of convolutional codes over a finite ring}\label{sec:coldis}

It is well-known that the free distance is the single most important parameter to determine
the performance of a block code. In the context of convolutional codes there are at least two
fundamental distance properties that are typically analysed, namely the free distance
and the column distance. In this section we formally introduce these two notions and
study convolutional codes that have optimal column distances. The free distance was studied in \cite{NaPiTo16,el13}.

  The {\bf weight} of $v(D)= \sum_{i \in \Z}{v_iD^i}\in \ZZ ((D))$ is given by
  $$
  \wt (v(D))=\sum_{i \in \Z}{\wt (v_i)}.
  $$

  The {\bf free distance} of a convolutional code $\C$ is defined as
  $$
  d(\C)=\min\{\wt (v(D)): \, v(D) \in \C, \, v(D) \neq 0\}.
  $$

\begin{Thm} \cite[Theorem 4.10]{el13}   \label{free_d}
The free distance of an $(n,k,\delta)$ convolutional code $\C$ satisfies

\begin{equation}\label{eq:singleton_bound}
d(\C) \leq n\left(\left\lfloor\frac{\delta}{k}\right\rfloor+1\right)-\left\lceil \frac{k}{r}\left(\left\lfloor\frac{\delta}{k}\right\rfloor+1\right)-\frac{\delta}{r}\right\rceil+1.
\end{equation}
\end{Thm}

Similarly to the field case, we call the bound (\ref{eq:singleton_bound}) the {\bf generalized Singleton bound}. As for column distance \cite{jo99} we define
$$
v(D)|_{[i,i+j]}=v_{i}D^{i}+  v_{i+1} D^{i + 1}+ \cdots + v_{i+j}D^{i+j}
$$
and analogously for $u(D)|_{[i,i+j]}$ for $u(D)=\sum_{\ell \in \Z}^{}u_{\ell}D^{\ell} \in \mathcal{A}_p^k ((D))$. The $j$-th \textbf{column distance} of a $p$-encoder $G(D)$ is defined as
\begin{eqnarray*}
  d^c_j(G(D)) &=&  \min\{\wt (v(D)|_{[i,i+j]}): \, v(D) =u(D)G(D), \,u_{i}\neq 0 \mbox{ and  }\,u_{\ell}= 0 \mbox{ for } \ell<i  \}\\
   &=&  \min\{\wt (v(D)|_{[0,j]}): \, v(D) =u(D)G(D), \,u_{0}\neq 0 \, \mbox{ and  }u_i=0 \mbox{ for } i<0    \} .
\end{eqnarray*}
This is a property of the $p$-encoder and different $p$-encoders can have different column distances. However, the column distances are invariant under the class of delay-free $p$-encoders of a code and it is equal to
\begin{eqnarray*}
  d^c_j(G(D)) &=&  \min\{\wt (v(D)|_{[i_{\min},i_{\min}+j]}): \ v(D) \in \C \},
\end{eqnarray*}
where  $v(D)=\sum_{\ell \geq i_{\min}} v_{\ell}D^{\ell} \in \ZZ^n((D))$ with $v_{i_{\min}}\neq 0$.
As every $(n,k,\delta)$-convolutional code $\cal C$ admits a delay-free $p$-encoder, we shall define the $j$-th \textbf{column distance} of $\cal C$,  denoted by $d^c_j(\C)$, as the column distance of one (and therefore all) of its delay-free $p$-encoders. If no confusion arises we use  $d^c_j$ for $d^c_j(\C)$. It is obvious that $d^c_j \leq d^c_{j+1}$ for $j \in \mathbb N_0$. \\

Next definition extends the well-known truncated sliding generator matrix of a convolutional code over a finite field \cite{gl03} to convolutional codes  over finite rings ($\ZZ$ in our case).

Given a $p$-encoder $G(D)=G_0+G_1D+ \dots +G_{\nu}D^\nu \in \Z^{k \times n}_{p^r}[D]$,
    we can define, for every $j \in \mathbb{N}_0$, the {\bf truncated sliding generator matrix} $G^c_j$ as
    $$
    G^c_j=
    \left[
      \begin{array}{cccc}
        G_{0} & G_{1} & \cdots  & G_{j}   \\
              & G_{0} & \cdots  & G_{j-1} \\
              &       & \ddots & \vdots   \\
              &       &        & G_{0}    \\
      \end{array}
    \right] \in \Z^{(j+1)k \times (j+1)n}_{p^r}
    $$
    where $G_\ell=0$ whenever $\ell >\nu$. 
In terms of the truncated sliding generator matrix the column distance reads as follows: Given a delay-free $p$-encoder $G(D)$ of a convolutional code $\C$ over $\ZZ$,
    the $j$-th column distance of $\C$ is given by
    $$
    d^c_j=\min\{\wt(v): v=u\, G^c_j \in \ZZ^{n(j+1)}, \; u=[u_{0} \dots u_{j}]\in \mathcal{A}_p^{k(j+1)}, \; u_{0} \neq 0 \}.
    $$
    for $j \in \mathbb{N}_0$.



Next, we present a result that allows to decompose a convolutional code over $\ZZ$ into simpler components.

\begin{Thm}\label{E1}
Every convolutional code $\C$ over $\ZZ$ admits a generator matrix of the form

\begin{equation}\label{eq:decomposition}
\widetilde{G}(D)= \left[
       \begin{array}{c}
         \widetilde{\mathcal{G}}_0(D) \\
         p\widetilde{\mathcal{G}}_1(D) \\
         \vdots \\
         p^{r-1}\widetilde{\mathcal{G}}_{r-1}(D) \\
       \end{array}
     \right],
\end{equation}
and such that

\begin{equation}\label{eq:decomposition_frr}
\widehat{G}(D)=\left[
       \begin{array}{c}
         \widetilde{\mathcal{G}}_0(D) \\
         \widetilde{\mathcal{G}}_1(D) \\
         \vdots \\
         \widetilde{\mathcal{G}}_{r-1}(D) \\
       \end{array}
     \right]
\end{equation}
is full row rank. Thus,
$$
\C_i:= \im_{\ZZ((D))} \, \widetilde{\mathcal{G}}_i(D)
$$
is a free convolutional code, $i=0, 1, \dots, r-1$, and
\begin{equation}\label{decomposition}
\C=\C_0 \oplus p \C_1 \oplus \dots \oplus p^{r-1} \C_{r-1}.
\end{equation}
\end{Thm}

\begin{proof}
Let $\widetilde{G}(D)$ be a generator matrix of $\C$. If $\widetilde{G}(D)$ is full row rank then $\C$ is free and $\C=\C_0$.\\
Let us assume now that $\widetilde{G}(D)$ is not full row rank. Then the projection of $\widetilde{G}(D)$ into $\Z_p[D]$,
$$
\overline{\widetilde{G}}(D) \in \Z_p^{k \times n}[D],
$$
is also not full row rank and there exists a nonsingular matrix $F_0(D) \in \Z_p^{k \times k}[D]$ such that
$$
F_0(D)\overline{\widetilde{G}}(D)=\left[ \begin{array}{c}
                                {\mathcal{G}}_0(D) \\
                                0
                             \end{array}
                      \right] \,\,\,\mbox{mod}\,\,\, p,$$
where ${\mathcal{G}}_0(D)$ is full row rank with rank $\ell_0$. Further, it follows that
$$
{F}_0(D)\widetilde{G}(D)=\left[\begin{array}{c}
                       \widetilde{\mathcal{G}}_0(D) \\
                       p \widehat{\mathcal{G}}_1(D)
                     \end{array}
                \right],
$$
where $\widetilde{\mathcal{G}}_0(D) \in \ZZ^{\ell_0 \times n}[D]$ is such that $\overline{\widetilde{\mathcal{G}}_0}(D)={\mathcal{G}}_0(D)$ and $\widehat{\mathcal{G}}_1(D) \in \ZZ^{(k-\ell_0) \times n}[D]$.
Moreover, since $F_0(D)$ is invertible, $\left[\begin{array}{c}\widetilde{\mathcal{G}}_0(D) \\p \widehat{\mathcal{G}}_1(D) \\\end{array}\right]$ is also a generator matrix of $\C$. Let us now consider $F_1(D) \in \Z_{p}^{(k-\ell_0) \times (k-\ell_0)}[D]$ such that
$$
F_1(D)\overline{\widehat{\mathcal{G}}_1}(D)=\left[\begin{array}{c}
                                          {\mathcal{G}}_1'(D) \\
                                          0 \\
                                        \end{array}
                                  \right]\,\,\, \mbox{mod}\,\,\, p,
$$
where ${\mathcal{G}'}_1(D)$ is full row rank with rank $\tilde \ell_1$ and
$$
{F}_1(D)\widehat{\mathcal{G}}_1(D)=\left[\begin{array}{c}
                                {\mathcal{G}}''_1(D) \\
                                 p \widehat{\mathcal{G}}_2(D)
                               \end{array}
                         \right],
$$
with $\mathcal{G}''_1(D) \in \ZZ^{\widetilde \ell_1 \times n}[D]$ such that $\overline{\mathcal{G}''_1}(D) = \mathcal{G}'_1(D)$ and $\widehat{\mathcal{G}}_2(D) \in \ZZ^{(k-\ell_0 - \widetilde \ell_1 ) \times n}[D]$. Hence,
$$
\left[
  \begin{array}{cc}
    I_{\ell_0} & 0 \\
    0 & F_1(D)
  \end{array}
\right]F_0(D)\widetilde{G}(D)=
\left[
  \begin{array}{c}
    \widetilde{\mathcal{G}}_0(D) \\
    p\mathcal{G}''_1(D) \\
    p^2\widehat{\mathcal{G}}_2(D)
  \end{array}
\right].
$$
If $\left[
  \begin{array}{c}
    \widetilde{\mathcal{G}}_0(D) \\
    \mathcal{G}''_1(D)
  \end{array}
\right]$ is not full row rank, then there exists a permutation matrix $P$ and a rational matrix $L_1(D) \in \ZZ^{\widetilde \ell_1 \times \ell_0}(D)$ such that
$$
P\left[
  \begin{array}{cc}
    I_{\ell_0} & 0 \\
    L_1(D) & I_{\widetilde \ell_1}
  \end{array}
\right]\left[
  \begin{array}{c}
    \widetilde{\mathcal{G}}_0(D) \\
    p\mathcal{G}''_1(D)
  \end{array}
\right]=
\left[
  \begin{array}{c}
    \widetilde{\mathcal{G}}_0(D) \\
    p\mathcal{G}'''_1(D) \\
    p^2\mathcal{G}'_2(D)
  \end{array}
\right],
$$
where $\mathcal{G}'''_1(D) \in \ZZ^{\ell_1 \times n}(D)$ and $\mathcal{G}'_2(D) \in \ZZ^{(\widetilde \ell_1-\ell_1) \times n}(D)$ are rational matrices and $\left[
  \begin{array}{c}
    \widetilde{\mathcal{G}}_0(D) \\
    \mathcal{G}'''_1(D)
  \end{array}
\right]$ is a full row rank rational matrix. Note that since
$$
P\left[
  \begin{array}{cc}
    I_{\ell_0} & 0 \\
    L_1(D) & I_{\widetilde \ell_1}
  \end{array}
\right]
$$
is nonsingular it follows that
$$\im_{\ZZ((D))} \left[
  \begin{array}{c}
    \widetilde{\mathcal{G}}_0(D) \\
    p\mathcal{G}''_1(D)
  \end{array}
\right] =  \im_{\ZZ((D))} \left[
  \begin{array}{c}
    \widetilde{\mathcal{G}}_0(D) \\
    p\mathcal{G}'''_1(D) \\
    p^2\mathcal{G}'_2(D)
  \end{array}
\right].$$
Let $\widetilde{\mathcal{G}}_1(D) \in \ZZ^{\ell_1 \times n}[D]$ and $\mathcal{G}''_2(D)\in \ZZ^{(\widetilde \ell_1-\ell_1) \times n}[D]$ be polynomial matrices (see Lemma \ref{rational encoder}) such that
$$
\im_{\ZZ((D))} \left[
  \begin{array}{c}
    \widetilde{\mathcal{G}}_0(D) \\
    p\mathcal{G}'''_1(D) \\
    p^2\mathcal{G}'_2(D)
  \end{array}
\right] =  \im_{\ZZ((D))} \left[
  \begin{array}{c}
    \widetilde{\mathcal{G}}_0(D) \\
    p\widetilde{\mathcal{G}}_1(D) \\
    p^2\mathcal{G}''_2(D)
  \end{array}
\right].
$$
Then
$
\left[
  \begin{array}{c}
    \widetilde{\mathcal{G}}_0(D) \\
    p\widetilde{\mathcal{G}}_1(D) \\
    p^2\mathcal{G}''_2(D)\\
    p^2\widehat{\mathcal{G}}_2(D)
  \end{array}
\right]
$
is still a generator matrix of $\C$ such that $ \left[ \begin{array}{c}
    \widetilde{\mathcal{G}}_0(D) \\
    \widetilde{\mathcal{G}}_1(D)
  \end{array}
\right] $ is full row rank.\\
Proceeding in the same way we conclude the proof.
\end{proof}


\begin{Rem}
The decomposition (\ref{decomposition}) could have been derived using the fact that $\ZZ^n((D))$ is a semi-simple module. Note, however, that Theorem \ref{E1} is constructive and its proof provides an algorithm to build the free modules $\C_i$. Moreover, it states that these submodules of $\ZZ^n((D))$ are indeed convolutional codes. Note that submodules of $\ZZ^n((D))$ do not always admit a polynomial or rational set of generators and therefore they are not necessarily convolutional codes.
\end{Rem}

If we denote by $\ell_i$ the rank of $\C_i$ then  $\{\ell_0,\ldots,\ell_{r-1}\}$ are clearly invariants of $\C$. We will call them the \textbf{parameters of the convolutional code }$\C$.

From now on, in order to simplify the exposition, we assume that the generator matrix $\widetilde{G}(D)$ is as in (\ref{eq:decomposition}) and such that $\widehat{G}(D)$ in (\ref{eq:decomposition_frr}) is such that $\widehat{G}(0)$ is full row rank. Hence, we can directly obtain a delay-free $p$-encoder by extending $\widehat{G}(D)$ as

$$
G(D)=\left[
\begin{array}{c}
\widetilde{\mathcal{G}}_0(D) \\
p \, \widetilde{\mathcal{G}}_0(D) \\
p \, \widetilde{\mathcal{G}}_1(D) \\
p^2 \, \widetilde{\mathcal{G}}_0(D) \\
p^2 \, \widetilde{\mathcal{G}}_1(D) \\
p^2 \widetilde{\mathcal{G}}_2(D) \\
\vdots \\
p^{r-1} \, \widetilde{\mathcal{G}}_0(D) \\
\vdots \\
p^{r-1} \widetilde{\mathcal{G}}_{r-1}(D) \\
\end{array}
\right]=\sum_{i \in \mathbb N_0} G_i D^i.
$$

As the rows of $G(0)= G_0$ form a $p$-basis (over $\ZZ$) then the parameters of the block code $\C_0 = \im_{\A} G(0)$ coincide with the parameters of $\C$.
Before establishing upper bounds on the column distances of a convolutional code we present a useful result on the truncated sliding matrix $G^c_j$ of $G(D)$. 

\begin{Prop}\label{lemma:01}
    If  $G(D) \in \ZZ^{k\times n}[D]$ is a $p$-encoder of a convolutional code $ \C$ then the rows of $G^c_j$ form a $p$-generator sequence, for any $j \in \mathbb{N}_0$.
\end{Prop}
\begin{proof}
See appendix.
\end{proof}


\begin{Thm} \label{dcj}
Let $\C$ be a $(n,k,\delta)$-convolutional code with parameters $k_0, k_1, \dots, k_{r-1}$. Then, it holds that
$$
d^c_j \leq (j+1)\left(n-\sum_{i=0}^{r-1}k_i\right)+1.
$$
\end{Thm}

\begin{proof}
  See appendix.
\end{proof}

The column distance measures the distance between two codewords within a time interval. Column distances are very appealing for sequential decoding: the larger the column distances the larger number of errors we can correct per time interval. Hence we seek for codes with column distances as large as possible.
Selecting an $r$-optimal set of parameters of a given $p$-dimension $k$, $(k_0, k_1, \dots, k_{r-1})$ the following corollary readily follows from Lemma \ref{opt_parameters}.

    \begin{Cor} \label{bound_dcj}
    Given a convolutional code $\C$  with length $n$ and $p$-dim$(\C)=k$ it holds
    $$
    d^c_j \leq \left(n-\left\lceil\frac{k}{r}\right\rceil\right)(j+1)+1.
    $$
    \end{Cor}

Let us denote the bound obtained in Corollary \ref{bound_dcj} for the column distance by
$$
B(j)=\left(n-\left\lceil\frac{k}{r}\right\rceil\right)(j+1)+1
$$
and the Singleton bound obtained in Theorem \ref{free_d} for the free distance by

\begin{eqnarray*}
SB &=& n\left(\left\lfloor\frac{\delta}{k}\right\rfloor+1\right)-\left\lceil \frac{k}{r}\left(\left\lfloor\frac{\delta}{k}\right\rfloor+1\right)-
\frac{\delta}{r}\right\rceil+1\\
   &=& \left(n-\frac{k}{r}\right)\left(\left\lfloor\frac{\delta}{k}\right\rfloor+1\right)+\frac{\delta}{r}-\varphi+1,
\end{eqnarray*}

with $\varphi=\left\lceil\frac{k}{r}\left(\left\lfloor\frac{\delta}{k}\right\rfloor+1\right)-
\frac{\delta}{r}\right\rceil-\left(\frac{k}{r}\left(\left\lfloor\frac{\delta}{k}\right\rfloor+1\right)-\frac{\delta}{r}\right)$.

Now we are in position to introduce the notion of maximum distance profile convolutional codes over a finite ring. These codes generalize the notion introduced in \cite{gl03} for maximum distance profile convolutional codes over finite fields to the ring case.

\begin{Def}
An $(n,k,\delta)$-convolutional code $\C$ over $\ZZ$ is said to be {\bf Maximum Distance Profile} (MDP) if
$$
d^c_j=B(j),
$$
for $j\leq L$, where
$$
L=max \{j:B(j) \leq SB\}.
$$
\end{Def}

A simple counting argument leads to the following result  which explicitly determines the value of such an $L$.

\begin{Thm} \label{L}
Let $\C$ be an MDP $(n,k,\delta)$-convolutional code over $\ZZ$ and
$$
X=\frac{\left(n-\frac{k}{r}\right)\left\lfloor\frac{\delta}{k}\right\rfloor+\frac{\delta}{r}-\varphi+
\left\lceil\frac{k}{r}\right\rceil-\frac{k}{r}}{n-\left\lceil\frac{k}{r}\right\rceil}
$$
with $\varphi=\left\lceil\frac{k}{r}\left(\left\lfloor\frac{\delta}{k}\right\rfloor+1\right)-
\frac{\delta}{r}\right\rceil-\left(\frac{k}{r}\left(\left\lfloor\frac{\delta}{k}\right\rfloor+1\right)-\frac{\delta}{r}\right)$. Then
$$
L= \left\lfloor X \right\rfloor.
$$
\end{Thm}
%

\section{Constructions of MDP Convolutional Codes over $\ZZ$}\label{sec:constructions}

In this section we will show the existence of MDP convolutional codes over $\ZZ$ for any given set of parameters $(n,k,\delta)$ such that $k \mid \delta$. Moreover, we will do that by building concrete constructions of such codes. In contrast with other existing constructions of convolutional codes over $\ZZ$ with designed distance \cite{no01,el13} where Hensel lifts of a cyclic code were used, we propose a method based on a direct \emph{lifting} of an MDP convolutional code from $\Z_p$ to $\ZZ$. We note that similar lifting techniques can be applied for different set of parameters $(n,k,\delta)$, see for more details \cite{TosteTese}. \\

Given the finite ring $\ZZ$ and the set of parameters $(n,k,\delta)$ with $k \mid \delta$, we aim to construct an MDP  $(n,k,\delta)$-convolutional code $\C$ over $\ZZ$. To this end, denote $k_0= \left\lfloor \frac{k}{r} \right \rfloor$ and $\nu= \frac{\delta}{k} $. Take $\widetilde{k}=k_0 + 1$ and $\widetilde{\delta}=\widetilde{k}\nu $, and let us consider an MDP convolutional code $\widetilde{\C}$ with length $n$, dimension $\widetilde{k}$ and degree $\widetilde{\delta}$ over $\Z_{p}$. Let $\widetilde{G}(D) \in \Z_p^{\widetilde{k} \times n}[D]$ be a minimal basic encoder of $\widetilde{\C}$, \emph{i.e.}, with $\widetilde{G}^{lc}$ full row rank over $\Z_p$ and left prime (constructions of such codes can be found in \cite{ AlmeidaNappPinto2013,gl03, NaRo2015}). Therefore,
\begin{eqnarray*}
  \widetilde{d^c_j} &=& \min\{\wt (v(D)|_{[0,j]}): \, v(D) =u(D)\widetilde{G}(D), \,u(D)=\sum_{i\in \mathbb{N}_0}u_i D^i \in \Z_p ((D)), \ u_{0}\neq 0    \} \\
   &=&  (j+1)(n- \widetilde{k})+1, \;\; j \leq \widetilde{L}
\end{eqnarray*}
where $\widetilde{L}=\left\lfloor \frac{\widetilde{\delta}}{\widetilde{k}} \right\rfloor + \left\lfloor \frac{\widetilde{\delta}}{n-\widetilde{k}}\right\rfloor$, see \cite{jo99,gl03}.

Let $R=k - k_0 r$ and decompose $\widetilde{G}(D)$  as

$$
\widetilde{G}(D)=\left[
\begin{array}{c}
\widetilde{\mathcal{G}}_{0}(D) \\
\widetilde{\mathcal{G}}_{{r-R}}(D)
\end{array}
\right]=\sum_{ 0 \leq i \leq \nu } \widetilde{G}_i D^i
$$
where $\widetilde{\mathcal{G}}_{k_0}(D)$ has $k_0$ rows and  $\widetilde{\mathcal{G}}_{k_{r-R}}(D)$ has $1$ row. In the case $r | k$ then $\widetilde{G}(D)=\widetilde{\mathcal{G}}_{0}(D)$. Next, we straightforward expand $\widetilde{G}(D)$ as

\begin{equation} \label{b}
G(D)=\left[
\begin{array}{c}
\widetilde{\mathcal{G}}_{0}(D) \\
p \, \widetilde{\mathcal{G}}_{0}(D) \\
\vdots \\
p^r \, \widetilde{\mathcal{G}}_{0}(D) \\
p^{r-R} \, \widetilde{\mathcal{G}}_{r-R}(D) \\
p^{r-R +1} \, \widetilde{\mathcal{G}}_{r-R}(D) \\
\vdots \\
p^{r-1} \, \widetilde{\mathcal{G}}_{r-R}(D)
\end{array}
    \right]=\sum_{ 0 \leq i \leq \nu } G_i D^i.
\end{equation}

Since $\widetilde{G}^{lc}$ is full row rank over $\Z_p$, it immediately follows that $G(D)$  is a $p$-encoder in reduced form.
\begin{Thm} \label{MDP}
Let $\mathcal{C}$ be a convolutional code over $\ZZ$ with $p$-encoder $G(D)$ as in (\ref{b}). Then, $\C$ is an MDP $(n,k,\delta)$-convolutional code over $\ZZ$.
\end{Thm}

\pf
 \ It is straightforward to verify that $\C$ is an $(n,k,\delta)$-convolutional code. It is left to show that it is an MDP code, \emph{i.e.}, we need to show that
$$
 d^c_j = \left(n-\left\lceil\frac{k}{r}\right\rceil\right)(j+1)+1.
$$
for $j \leq L$ as in Theorem \ref{L}. It is a matter of straightforward computations to verify that since $k \mid \delta$, $L= \widetilde{L} = \left\lfloor \frac{\widetilde{\delta}}{\widetilde{k}} \right\rfloor + \left\lfloor \frac{\widetilde{\delta}}{n-\widetilde{k}}\right\rfloor$ . The $j$-th truncated sliding matrix correspondent to $G(D)$ is

$$
G^c_j=\left[
     \begin{array}{ccccc}
       G_0 & G_1 & \dots  &  G_{j-1} & G_j \\
           &     & \ddots & \vdots   & \vdots \\
           &     &        & G_0      & G_1 \\
           &     &        &          & G_0 \\
     \end{array}
   \right]
$$

Let $u=\left[
    \begin{array}{cccc}
      u_0 & u_1 & \dots & u_j \\
    \end{array}
  \right],
$ with $u_i \in \A^{k}$, $i=0, \dots, j$ and $u_0\neq 0$, and let
$$
v=\left[
    \begin{array}{cccc}
      v_0 & v_1 & \dots & v_j \\
    \end{array}
  \right],
$$
with $v_i \in \Z_p^n$,  $i=0, \dots, j$, such that
$
v=uG^c_j.
$
The ideia of the proof is to multiply $v$ by a power of $p$ such that the resulting nonzero truncated codeword $\widetilde v$ is in $p^{r-1}\ZZ^n$. Since $p^{r-1}\ZZ$ is isomorphic to $\Z_p$ then there exists a truncated nonzero codeword $\widehat v \in \widetilde \C = \im_{\Z_p ((D))} \widetilde G(D)$ such that $\wt(\widehat v)=\wt(\widetilde v)$, and then we can use the fact that $\widetilde \C$ is MDP.

 We define the {\bf order} of $v$, denoted by $ord(v)$, as the $j\in \{ 1,2,\dots, r \}$ such that $p^j v=0 \mbox{ and } p^{j-1} v \neq 0.$ Take
$$
\ell=\max_{\begin{array}{c}
              0 \leq t \leq j
         \end{array}}
{ord(v_t)}
$$
and
$$
i=\min_{\begin{array}{c}
            0 \leq s \leq j
          \end{array}}\{s: ord(v_s)=\ell\}
   =\min_{\begin{array}{c}
            0 \leq s \leq j
          \end{array}}\{s: p^{\ell-1}v_s \neq 0\}.
$$
There exists $\widehat{v}_s \in \A^n$ such that $\tilde{v}_s=p^{\ell-1}v_s=p^{r-1}\hat{v}_s,$ $s=i, \dots, j$ and then

\begin{equation}
\begin{split}
p^{\ell-1}v & =
\left[
  \begin{array}{ccccccc}
    0 & 0 & \dots & 0 & \widetilde{v}_i & \dots & \widetilde{v}_j \\
  \end{array}
\right] =
p^{r-1}
\left[
  \begin{array}{ccccccc}
    0 & 0 & \dots & 0 & \widehat{v}_i & \dots & \widehat{v}_j \\
  \end{array}
\right].
\end{split}
\end{equation}

Now it can be easily checked that

$$
p^{\ell-1}v= p^{r-1} \left[
                       \begin{array}{cccccc}
                         \widetilde{u}_0 & \widetilde{u}_1 & \dots & \widetilde{u}_i & \dots & \widetilde{u}_j \\
                       \end{array}
                     \right]
                    \left[
                       \begin{array}{cccccc}
                         \widetilde{G}_0 & \widetilde{G}_1 & \dots  & \widetilde{G}_i     & \dots  & \widetilde{G}_j \\
                                     & \widetilde{G}_0     & \dots  & \widetilde{G}_{i-1} & \dots  & \widetilde{G}_{j-1} \\
                                     &                     & \ddots & \vdots              &        & \vdots \\
                                     &             &       & \widetilde{G}_0              & \dots  & \widetilde{G}_{j-i} \\
                                     &             &       &                              & \ddots & \vdots \\
                                     &             &       &                              &        & \widetilde{G}_0 \\
                       \end{array}
                     \right],
$$
for some $\widetilde{u}_0, \widetilde{u}_1, \dots, \widetilde{u}_i, \dots, \widetilde{u}_j \in \A^{\widetilde{k}}$,
with $\widetilde{u}_0=\dots=\widetilde{u}_{i-1}=0$, because $\widetilde{G}_0$ is full row rank and therefore,
$$
\left[
  \begin{array}{ccc}
    \widetilde{v}_i & \dots & \widetilde{v}_j \\
  \end{array}
\right]=p^{r-1}
\left[
  \begin{array}{ccc}
     \widetilde{u}_i & \dots & \widetilde{u}_j \\
  \end{array}
\right]
\left[
  \begin{array}{ccc}
    \widetilde{G}_0 & \dots  & \widetilde{G}_{j-i} \\
                    & \ddots & \vdots \\
                    &        & \widetilde{G}_0 \\
  \end{array}
\right]
$$
where $\widetilde{u}_i \neq 0$.
Using the fact that $\widetilde{\C}=\im_{\Z_p[D]}\widetilde{G}(D)$ is MDP we obtain
\begin{equation*}
  \begin{split}
  \wt\left(\left[
  \begin{array}{ccc}
    v_i & \dots & v_j \\
  \end{array}
\right]\right) &
\geq \wt\left(\left[
  \begin{array}{ccc}
    \widetilde{v}_i & \dots & \widetilde{v}_j \\
  \end{array}
\right]\right) \geq  (n- \widetilde{k})(j-i+1)+1.
  \end{split}
  \end{equation*}

Considering
$
\left[
  \begin{array}{ccc}
    v_0 & \dots & v_{i-1}\\
  \end{array}
\right]=
\left[
  \begin{array}{ccc}
    u_0 & \dots & u_{i-1}\\
  \end{array}
\right]G^c_i
$
and reasoning in the same way we conclude that
$$
\wt\left(\left[
  \begin{array}{ccc}
    v_0 & \cdots &  v_{i-1} \\
  \end{array}
\right]\right) \geq  (n- \widetilde{k})i+1
$$
and therefore
$$
\wt\left(\left[
  \begin{array}{ccc}
    v_0 & \cdots & v_j \\
  \end{array}
\right]\right) \geq  (n- \widetilde{k})(j+1)+1.
$$
Consequently, $d^c_j=(n- \widetilde{k})(j+1)+1$, \emph{i.e.},
$$
d^c_j=(n- \left\lceil\frac{k}{r}\right\rceil)(j+1)+1,
$$
for $j \leq L$.
\eind \pfend

\section{Appendix}\label{appendix}

\textbf{\emph{Proof of Proposition \ref{lemma:01}}:} Let us represent $G(D)$ by
$$
    G(D)=
    \left[
      \begin{array}{c}
        g_1(D) \\ g_2(D) \\ \vdots \\ g_k(D)
      \end{array}
    \right]
$$

\noindent where $g_s(D)= \displaystyle \sum_{i \in \mathbb N_0 }g^i_sD^i$, with $s=1, \dots, k$, is the $s$-th row of $G(D)$.
Since $G(D)$ is a $p$-encoder, its rows form a $p$-generator sequence and therefore

\begin{enumerate}
      \item $p\, g_s(D) \in p\mbox{-}span (g_{s+1}(D), \dots, g_k(D))$, $s=1, \dots, k-1;$
      \item $p\, g_k(D)=0$.
\end{enumerate}

Thus, $p\, g_s(0) \in p\mbox{-}span (g_{s+1}(0), \dots, g_k(0))$, $s=1, \dots, k-1,$ and $p\, g_k(0)=0$, which means that the rows of $G^c_0$ form a $p$-generator sequence.

Let us assume now that the rows of $G^c_{j}$ form a $p$-generator sequence and let us prove that the rows of $G^c_{j+1}$ also form a $p$-generator sequence.
For that it is enough to prove that
\begin{equation}\label{T}
  p \,row_s(G^c_{j+1}) \in p\mbox{-}span (row_{s+1}(G^c_{j+1}), \dots, row_{k(j+1)}(G^c_{j+1})),
\end{equation}

$s=1, \dots, k$, where $row_{i}(G^c_{j+1})$ denotes the $i$-th row of $G^c_{j+1}$.

Let $s \in \{1, \dots, k-1 \}$. By condition 1)  there exists $
    a_t(D)=\displaystyle \sum_{i \in \mathbb N_0 }a^i_t D^i \in \mathcal {A}_p[D]$, $t=s+1, \dots, k$, such that
\begin{eqnarray*}
  p \, g_s(D) &=& a_{s+1}(D)\cdot g_{s+1}(D)+ a_{s+2}(D)\cdot g_{s+2}(D)+\dots +a_k(D)\cdot g_k(D) \\
\end{eqnarray*}
which implies that

\begin{eqnarray*}
    p \, \left [\begin{array}{cccc}
                     g^0_s & g^1_s & \cdots & g^{j+1}_s
                   \end{array} \right]
    &=&
    a^0_{s+1} \cdot \left[\begin{array}{cccc}
                     g^0_{s+1} & g^1_{s+1} & \cdots & g^{j+1}_{s+1}
                   \end{array} \right] + \cdots +\\
    & &+ a^0_{k} \, \left[\begin{array}{cccc}
                     g^0_{k} & g^1_{k} & \cdots & g^{j+1}_{k}
                   \end{array} \right]+ \\
    & &+a^1_{s+1} \, \left[\begin{array}{cccc}
                     0 & g^0_{s+1} & \cdots & g^{j}_{s+1}
                   \end{array} \right] + \cdots
    a^1_{k} \, \left[\begin{array}{cccc}
                     0 & g^0_{k} & \cdots & g^{j}_{k}
                   \end{array} \right] + \\
    & & + \cdots +\\
    & & + a^{j+1}_{s+1} \, \left[\begin{array}{cccc}
                     0 & \cdots & 0 &  g^{0}_{s+1}
                   \end{array} \right] + \cdots +
    a^{j+1}_{k} \, \left[\begin{array}{cccc}
                     0 & \cdots & 0 & g^{0}_{k}
                   \end{array} \right],
\end{eqnarray*}
which proves (\ref{T}). Finally, let us consider now $s=k$. Since the rows of $G(D)$ form a $p$-generator sequence, $p \, g_k (D)=0$ and therefore $p \, row_k(G_{j+1}^c)=0$.
    \eind \pfend

\textbf{\emph{Proof of Theorem \ref{dcj}}}:
Let $\widetilde G(D) \in \ZZ^{k \times n}[D]$ be a generator matrix of $\C$ as in (\ref{eq:decomposition}) with $\widehat{G}(D)$ in (\ref{eq:decomposition_frr}) full row rank and such that $\widehat{G}(0)$ is also full row rank. Let us consider the $p$-encoder
$$G(D)=\left[
\begin{array}{c}
\widetilde{\mathcal{G}}_0(D) \\
p \, \widetilde{\mathcal{G}}_0(D) \\
p \, \widetilde{\mathcal{G}}_1(D) \\
p^2 \, \widetilde{\mathcal{G}}_0(D) \\
p^2 \, \widetilde{\mathcal{G}}_1(D) \\
p^2 \widetilde{\mathcal{G}}_2(D) \\
\vdots \\
p^{r-1} \, \widetilde{\mathcal{G}}_0(D) \\
\vdots \\
p^{r-1} \widetilde{\mathcal{G}}_{r-1}(D) \\
\end{array}
\right]=\sum_{i \in \mathbb N_0} G_i D^i.
$$
Since $\widehat{G}(0)$ is full row rank, $G(D)$ is delay-free. Moreover, the last $k_0+k_1+ \cdots + k_{r-1}$ rows of $G(D)$ belong to $p^{r-1}\mathbb Z_{p^r}^n[D]$ which implies that the last $k_0+k_1+ \cdots + k_{r-1}$ rows of $G_i$ belong to $p^{r-1}\mathbb Z_{p^r}^n$, for all $i$.

Let us consider the truncated sliding generator matrix $G^c_j$ to obtain
$$
d^c_j=d^c_j(G)=\min\{\wt(v): v=u\,  G^c_j , u=[u_0 \dots u_j], u_0 \neq 0, u_i \in \mathcal{A}_p^k, i=0, \dots, j\}.
$$
We can assume without loss of generality that $G_0$ is in $p$-standard form as in (\ref{p-standG0}),
with parameters $k_0, k_1, \dots, k_{r-1}$.

Consider $u=\left[
  \begin{array}{cccc}
   u_0 & u_1 & \cdots & u_j \\
   \end{array}
\right], \;\; u_i \in \mathcal{A}^k_p$, $i=0,\dots,j$ with $u_0=
\left[
  \begin{array}{ccccc}
    0 & 0 & \dots & 0 & 1 \\
   \end{array}
\right]
$ and $v=u G^c_j = \left[\begin{array}{cccc}
   v_0 & v_1 & \cdots & v_j
   \end{array}
\right]$ with $v_i \in \mathbb{Z}^n_{p^r}$, $i=0,\dots,j$.

Then $v_0=u_0  G_0 =
    \left[
      \begin{array}{ccccc}
        0 & \dots & 0 & 1 & p^{r-1}A^{r-1,k}_{r,r-1} \\
      \end{array}
    \right],
$
where $A^{r-1,k}_{r,r-1}$ represents the last row of $A^{r-1}_{r,r-1}$ as in (\ref{p-standG0}). Then
$$
    \wt(v_0) \leq n-(k_0+k_1+ \dots +k_{r-1})+1.
$$
Note that $v_1=p^{r-1}g_1 + u_iG_0$ where $p^{r-1}g_1$ represents the last row of $G_1$.
Write $g_1$ as
$$
    g_1=
    \left[
      \begin{array}{ccccc}
        g_{1,k_0} & g_{1,k_1} & \dots & g_{1,k_{r-1}} & g_{1,n-(k_0+ \dots +k_{r-1})} \\
      \end{array}
    \right],
$$
with $ g_{1,i} \in \ZZ^i$, $i=k_0, k_1, \dots, k_{r-1}$ and
$g_{1,n-(k_0+ \dots +k_{r-1})} \in \ZZ^{n-(k_0+ \dots +k_{r-1})}$.\\

Let us consider $u_1$ with:
\begin{description}
      \item[-] its first $[(r-1)k_0+(r-2)k_1+ \dots + k_{r-2}]$ components equal to zero;
      \item[-] the remaining $k_0+k_1+ \dots + k_{r-1}$ components equal to
    $$\left[
      \begin{array}{cccc}
        \alpha_{1,k_0} & \alpha_{1,k_1} & \cdots & \alpha_{1,k_{r-1}} \\
      \end{array}
    \right],
    $$
    where $\alpha_{1,k_i} \in\A^i$ are such that $-p^{r-1} \, g_{1,k_i}=p^{r-1}\alpha_{1,k_i}$, $i=0, \dots, r-1$.
\end{description}

So, we obtain $v_1$ with its first $(k_0+k_1+ \dots + k_{r-1})$ elements equal to zero, and therefore
\begin{eqnarray*}
      \wt(v_1) & \leq & n-(k_0+k_1+ \dots + k_{r-1}).
\end{eqnarray*}
In the same way,
$$
    v_2=p^{r-1} \, g_2+u_1 G_1+u_2 G_0
$$
where $ p^{r-1} \, g_2$ represent the last row
of $G_2$ and $u_1 G_1 \in p^{r-1}\ZZN$. Take $u_2$ such that:
\begin{description}
  \item[-] its first $[(r-1)k_0+(r-2)k_1+ \dots + k_{r-2}]$ components are zero;
  \item[-] the remaining $(k_0+k_1+ \dots + k_{r-1})$ components are equal to
 $$\left[
      \begin{array}{cccc}
        \alpha_{2,k_0} & \alpha_{2,k_1} & \cdots & \alpha_{2,k_{r-1}} \\
      \end{array}
    \right],
    $$
    where $\alpha_{2,k_i} \in\A^i$ are such that $-p^{r-1} \, \tilde g_{2,k_i}=p^{r-1}\alpha_{2,k_i}$, $i=0, \dots, r-1$, where $\left[
      \begin{array}{cccc}
        p^{r-1} g_{2,k_0} & \ p^{r-1} g_{2,k_0} & \cdots &  p^{r-1} g_{2,k_{r-1}}
      \end{array}
    \right]$ represent the first $k_0+k_1+ \cdots + k_{r-1}$ components of $p^{r-2}g_2 + u_1 G_1$.
\end{description}
As before, the first $k_0+k_1 + \cdots + k_{r-1}$ elements of $v_2$ are zero and therefore
$$
\wt(v_2) \leq n-(k_0+k_1+ \dots + k_{r-1}).
$$

Applying the same reasoning we construct $u_i \in \A^k$ such that $\wt(v_i) \leq n-(k_0+k_1+ \dots + k_{r-1})$, $i=3, \dots,j$ and therefore
\begin{equation*}
d^c_j \leq (j+1)n-(j+1)(k_0+k_1+ \dots + k_{r-1})+1.
\end{equation*}\eind \pfend

\bibliographystyle{plain}


\bibliography{biblio_com_tudo}
\end{document}